\documentclass[11pt]{article} 

\usepackage[utf8]{inputenc} 

\usepackage{geometry} 
\usepackage{color}
\usepackage{graphicx} 

\usepackage{array} 

\usepackage{stmaryrd}
\usepackage[titletoc]{appendix}
\usepackage{mathtools}
\usepackage{amsfonts,hyperref,url}
\usepackage{amssymb}
\usepackage{amsthm}
\usepackage{ upgreek }
\usepackage{bbm}
\usepackage{esint}

\newcommand{\cadlag}{c\`adl\`ag }

\usepackage{mathrsfs}
\usepackage{bm}

\usepackage{appendix}

\newcommand{\eq}[1]{\begin{align}#1\end{align}}

\newcommand{\ba}{\begin{eqnarray}}
\newcommand{\ea}{\end{eqnarray}}
\newcommand{\epsi}{\varepsilon}
\def\d{{\rm d}}

\newtheorem{theorem}{Theorem}
\numberwithin{theorem}{section}

\newtheorem{lemma}[theorem]{Lemma}
\newtheorem{corollary}[theorem]{Corollary}

\newtheorem{proposition}[theorem]{Proposition}

\newtheorem*{theorem*}{Theorem}
\newtheorem*{lemma*}{Lemma}
\newtheorem*{corollary*}{Corollary}
\newtheorem*{proposition*}{Proposition}
\newtheorem*{problem*}{Problem}
\newtheorem*{conjecture*}{Conjecture}

\theoremstyle{definition}
\newtheorem{definition}[theorem]{Definition}
\newtheorem{hyp}{Assumption}
\numberwithin{hyp}{section}
\newtheorem{claim}{Property}

\newtheorem{example}[theorem]{Example}

\theoremstyle{remark}
\newtheorem{remark}[theorem]{Remark}

\newcommand{\Rr}{{\mathbb{R}}}

\everymath{\displaystyle}

\begin{document}

	\title{Rough differential equations with\\
		path-dependent coefficients}

	\author{Anna Ananova\\ Mathematical Institute\\
		University of Oxford.\\
		{\tt Anna.Ananova@maths.ox.ac.uk}
	}

	\date{}

	\maketitle
	
	\begin{abstract}
		We establish the existence of solutions to path-dependent rough differential equations with non-anticipative coefficients. Regularity assumptions on the coefficients are formulated in terms of horizontal and vertical derivatives.
	\end{abstract}

	\tableofcontents

	\section{Introduction}
	
	The theory of rough paths \cite{lyons} provides a  framework for defining  solutions to differential equations driven by a rough signal:
	\begin{equation} dY(t)= V(Y(t)) dX(t)\qquad X(0)=x\in \mathbb{R}^d \label{eq.ODE}\end{equation}
	where $V$ is a smooth vector field and $X, Y$ are continuous, but non-smooth functions whose lack of regularity prevents an interpretation of \eqref{eq.ODE} in terms of Riemann-Stieltjes or Young integration. A key insight of T. Lyons  \cite{lyons98} was to supplement the signal $X$ with a {\it rough path} tensor $\mathbb{X}$ constructed above $X$ such that one can construct an integration theory for \eqref{eq.ODE} with respect to the enriched path $(X,\mathbb{X})$. One of the main results of the theory is that \eqref{eq.ODE} maybe then interpreted as a `rough differential equation' (RDE) using this notion of rough integration.
	
	Since the pioneering work of Lyons \cite{lyons98,lyons2007}, the study of such rough differential equations (RDEs) has developed in various directions.
	Solutions to rough differential equations have been constructed as limits of discrete approximations \cite{davie2008},
	fixed points of Picard iterations  \cite{lyons98,gubinelli2004,lyons2007,frizhairer} or 
	limits of solutions of certain ODEs \cite{bailleul2015,frizvictoir}.
	An essential technique underlying many results is  Picard iteration in the Banach space of
	controlled paths, introduced by Gubinelli \cite{gubinelli2004} in the case of   H\H older regularity $\alpha\in (1/3,1/2]$ and extended to the case of arbitrary regularity in \cite{gubinelli2010,hairer2014}.

	As shown by Lyons \cite{lyons98}, for stochastic differential equations driven by Brownian motion,  probabilistic (Stratonovich) solutions coincide with RDE solutions constructed for an appropriate choice of Brownian rough path, showing that the theory of RDEs is also relevant for the study of stochastic differential equations (SDEs). 
	SDEs with path-dependent features arise in many problems in stochastic analysis and stochastic modeling \cite{ItoNis, Kush, Moh}, and this natural link with RDEs has inspired several studies on rough differential equations with path-dependent features which echo examples of path-dependent SDEs encountered in stochastic models \cite{aida2019,bailleul2014,deya2019,neuenkirch2008}. 
	
	
	A classical technique used in the study of path-dependent stochastic equations is to lift them to an infinite-dimensional SDE in the space of paths \cite{Moh}. This approach has been adapted to RDEs in Banach spaces by Bailleul \cite{bailleul2014} but requires Fr\'echet differentiability of the vector fields (coefficients), an assumption which excludes many examples.  Neuenkirch et al. \cite{neuenkirch2008} show existence and uniqueness for RDEs with delay;  
	Aida \cite{aida2019} and Deya et al. \cite{deya2019} study a class of  RDEs with path-dependent bounded variation terms, motivated by reflected SDEs. Although these examples may be represented as Banach space-valued RDEs, the functional coefficients involved fail to have sufficient regularity to apply the Banach space approach  \cite{bailleul2014}, and the results in \cite{aida2019,deya2019,neuenkirch2008} are specific to the class of equations considered.

	In this study, we complement these results by revisiting the existence of solutions for a class of path-dependent RDEs using a weaker notion of regularity, based on the non-anticipative functional calculus introduced in \cite{CF09,cont2012, D09}. This functional calculus is based on certain directional derivatives and does not require Fr\'echet differentiability, covering a larger class of examples of ODEs and SDEs with path-dependent coefficients  \cite{contkalinin2019}.
	
	We consider path-dependent rough  differential equations (RDEs) whose coefficients  are non-anticipative functionals \begin{equation}
	\begin{cases}
	dY(s)= b(s, Y_s)ds+\sigma(s,Y_s)dX (s),\label{eq.RDE}\\
	Y_{t_0}=\xi_{t_0}.
	\end{cases}
	\end{equation}
	where $\bm{X}=(X,\mathbb{X})$ is a $p$-variation rough path with $p\in[2,3)$ and $b,\sigma$ are non-anticipative functionals allowing for dependence on the (stopped) path $Y_s=Y(s\wedge .)$.
	We define regularity conditions on the coefficients in terms of the existence and continuity of functional derivatives in the sense of Dupire \cite{cont2012, D09}; these conditions are much weaker than Fr\'echet differentiability and only involve certain directional derivatives.
	
	As in \cite{frizhairer}, a solution of \eqref{eq.RDE} is defined as a controlled path $(Y, Y')$ such that $Y'(s)=\sigma(s,Y_s)$ and 
	\[
	Y(t)=\xi_{t_0}\,+\, \int_{t_0}^t b(s, Y)\, \d s\,+\,\int_{t_0}^t \sigma(s, Y) d\bm{X}.
	\]
	where the second integral is a rough integral.
	Our main result is an existence theorem  (Theorem \ref{chp3theorem:existence}) for solutions to  \eqref{eq.RDE}.
	Detailed definitions, assumptions on coefficients, and precise statements of results are presented below.
	The proof is based on an adaptation of the proof of  Peano's existence theorem 
	\cite{peano} to this setting and a fixed point argument  for the map
	\[
	(Y, Y')\mapsto (\xi_0+\int_0^{\cdot} b(t, Y)dt+\int_0^{\cdot} \sigma(\cdot, Y) d\bm{X} ,\, \sigma(\cdot, Y) ).
	\]
	The main difficulty is to obtain estimates on this map, given the path-dependence in the coefficients.
	
	\paragraph{Outline}  Section \ref{chp:introrough} provides an overview of rough path theory and controlled paths, and recalls the definition of the rough integral and its basic properties. In Section \ref{chp:actions}, we prove several  results on the action of regular functionals on rough paths and controlled paths: Lemmas \ref{chp3theorem:controll}, \ref{chp3theorem:actiononpaths} and Theorem \ref{chp3theorem:roughpathint}. Finally, section \ref{chp:existence} presents the setting of the problem and our main result on the existence of  solutions to path-dependent RDEs (Theorem \ref{chp3theorem:existence}).
	
	\vskip0.5cm
	{\small {\bf Acknowledgements.}
		We thank Rama Cont for fruitful discussions and valuable suggestions that helped to improve the article.}
	
	\section{Rough paths and rough integration}
	\label{chp:introrough}

	We begin by recalling some concepts from the theory of rough paths \cite{frizhairer, lyons98, lyons2007}. We will focus on the simplest of continuous paths $X$ with finite $p$-variation, for $p\in [2, 3)$.

	\begin{definition}[{p-variation paths}]
		We denote by $C^{p-var}([0,T],\Rr^d)$ the set of continuous paths $X\in C([0,T];\Rr^d)$, such that
		\[
		\|X\|_{p, [0,T]}:=
		\left(\sup_{\pi\in\mathcal{P}([0,T])}\sum_{t_k\in \pi} |X(t_{k+1})-X(t_k)|^p\right)^{\frac 1 p}<+\infty.
		\]
		where the supremum is taken over the set   $\mathcal{P}([0,T])$ of all partitions of the interval $[0,T].$ 
		Similarly for functions of two variables $R_{\cdot , \cdot}\colon[0,T]^2\to \Rr^d$, we define
		\[
		\|R\|_{p, [0,T]}=\left(\sup_{\pi\in\mathcal{P}([0,T])}\sum_{t_k\in \pi} |R_{t_k, t_{k+1}}|^p\right)^{\frac 1 p}.
		\]
	\end{definition}
	We denote by $V_p(X;t,s)$ the $p$-variation of the path ${X\in}C^{p-var}([0,T],\Rr^d)$ on the interval $[t,s]$:
	\[
	V_p(X;t,s):=\|X\|^p_{p, [t,s]}.
	\]
	One obviously has
	\ba\label{chp2ieq:varholder}
	|X(s)-X(t)|^p\leq V_p(X; t, s).
	\ea
	and $V_p$ is superadditive:
	\ba\label{chp2ieq:varadditive}
	V_p(X; t,u)+V_p(X; u,s)\leq V_p(X;t,s),\, \forall t\leq u\leq s.
	\ea
	As a consequence the function $V_p(X; 0, \cdot)$ is increasing  and continuous.
	
	The above motivates the notion of a superadditive function on the set of the intervals:
	\begin{definition}[Superadditive interval functions]
		A map $$\omega\colon \{ [t,s]\colon 0\leq t\leq s\leq T \}\to \Rr_+,$$ with $\omega[t,t]=0, \forall t\in[0,T]$ is called superadditive if for all $t\leq u\leq s$ in $[0, T]$
		\[
		\omega([t,u])+\omega([u,s])\leq \omega([t,s]).
		\]
	\end{definition}
	A basic example of a superadditive function is $\omega([t,s])=\|X\|^p_{p,[t,s]}$ for any $X\in C^{p-var}([0,T],\Rr^d).$
	A very useful fact about superadditive functions, which will be used in the paper, is that for $\omega_1, \omega_2$ superadditive, so are  $\omega_1^{r}$ and $\omega_1^{\theta}\omega_2^{1-\theta},$ for all $r\geq 1,\, \theta\in (0,1).$
	
	The notion of superadditive functions allows us to formulate an alternative definition of the space of $p$-variation paths:
	\begin{proposition*}
		$X\in C^{p-var}([0,T],\Rr^d)$ if and only if there exists a superadditive function $\omega$ such that $$|X(s)-X(t)|\leq \omega([t,s])^{\frac 1 p},\, \forall t\leq s\in[0,T].$$
	\end{proposition*} 
	The above definition is closer to the definition of H\H older continuous paths, and corresponds to the latter in the case $\omega([t,s])=|s-t|.$

	We now define the space of rough paths (see e.g. \cite{FritzShekhar}[Sec. 1.2.4]):
	\begin{definition}[Space of $p$-rough paths]\label{def.roughpath}
		For $p\in[2, 3)$ we define the space $\mathcal{C}^{p-var}([0,T],\Rr^d) $ of continuous $p$-rough paths as the set of pairs $\mathbf{X}:=(X,\mathbb{X})$ of {$\Rr^d\times\Rr^{d\times d}$-valued} paths such that 
		\begin{enumerate}
			\item[(i)]     \[
			\mathbb{X}_{t,u}-\mathbb{X}_{t,s}-\mathbb{X}_{s,u}=X_{t,s}\otimes X_{s,u},\quad \forall  t,s,u\in[0,T].
			\] 
			\item[(ii)]
			\[
			\|X\|_{p, [0,T]}+ \|\mathbb{X}\|_{{\frac p 2}, [0,T]}<+\infty.
			\]
		\end{enumerate}
		
	\end{definition}
	As shown by Lyons and Victoir  \cite{LV07}, any H\H older continuous path {$X\in C^\alpha([0,T],\Rr^d)$ can} be associated with a rough path, but this association is far from canonical and in fact for $\alpha<1/2$ there are infinitely many such rough paths.

	Now, we define the analog of weakly controlled paths \cite[Def.1]{gubinelli2004}:
	\begin{definition}[Controlled paths]\label{chp02def:conrtollab}
		Let {$p>q>1$ }and $X\in C^{p-var}([0,T],\Rr^d)$. A pair $(Y,Y')\in C^{p-var}([0,T],\Rr^k)\times C^{p-var}([0,T],\Rr^{d\times k})$ of finite $p$-variation a $(p, q)$-{controlled path} with respect to $X$ if
		\[
		R^Y_{t,s} {:=} Y_{t, s} - Y'_t X_{t, s}
		\]
		has a finite $q$-variation. We denote by $\mathcal{D}^{p, q}_{_X}([0,T], \Rr^k)$ the set of all {$(p, q)$}-{controlled path}s with respect to $X$. 
	\end{definition}
	The path $X$ is called the {\it control} or {\it reference path}. 
	Typical examples of {controlled path}s arise from smooth functions $f: \Rr^d\to \Rr$ of $X$:
	$$ Y(t)=f(X(t)),\qquad Y'(t)=\nabla f(X(t)).$$
	$R^Y(t,s)$ is then  given by the remainder in a first order Taylor expansion. By analogy any $Y'$ satisfying  Def. \ref{def.roughpath} is called a `Gubinelli derivative' for $Y$.  $R^{Y}(s,t)$ plays the role of a remainder in a first order  expansion of $Y$, and $Y'(s)$ plays the role of a  `derivative' of $Y$ with respect to $X$. The requirement is that the remainder $R^{Y}$ is smoother than $Y$ itself: we go from $p$ to $q<p$ in the finite variation regularity scale.
	The above definition corresponds to weakly-controlled paths in \cite{gubinelli2004}, for our convenience, throughout the paper we will use the name ``$(p, q)-$ controlled paths'' or ``controlled paths'' if the exponents $p, q$ are apparent from the context. 
	
	One can check that  $\mathscr{D}^{p, q}_{X}([0,T],\Rr^d)$ is a Banach space under the norm 
	$$\|(Y,Y')\|_{\mathscr{D}^{p, q}_{X}}=    |Y_0|+|Y'_0|+ \underbrace{\|Y'\|_{p, [0,T]}\,+\,\|R^Y\|_{q, [0,T]}}_{:=\|Y,Y'\|_{p, q,X}}.$$

	The next theorem establishes that controlled paths are proper integrands for rough integration:
	\begin{theorem}[c.f. {\cite{frizhairer}[Theorem 4.10]}, {\cite{gubinelli2004}[Theorem 1]}]\label{chp02theorem:itegralcontrol}
		Let  $p\in [2, 3),\, q\geq p/2$,  $\mathbf{X}=(X,\mathbb{X})\in\mathcal{C}^{p-var}([0,T], \Rr^d)$. Let also $p^{-1}+q^{-1}>1$ and $(Y,Y')\in\mathcal{D}^{p, q}_{_X}([0,T], \Rr^k)$ be a {controlled path}. Define the compensated Riemann sums
		\[
		S (\pi):=\sum_{[t, s]\in \pi} Y_{t}X_{t, s}+ Y'_{t}\mathbb{X}_{t,s}.
		\]
		Then the limit
		\[
		\int_0^T Y d\mathbf{X}:= \lim_{|\pi|\to 0} S (\pi)
		\]
		exist and satisfies the estimate
		\ba
		\left|\int_t^s Y dX-Y_{t}X_{t, s}- Y'_{t}{\mathbb{X}}_{t,s}\right|\nonumber
		\leq C\left(\|{R}^Y\|_{q,[t,s]}\|X\|_{p,[t,s]}+\|Y'\|_{p,[t,s]}\|{\mathbb{X}}\|_{p/2,[t,s]}\right)\nonumber.
		\ea
		Moreover, the map
		\[
		(Y,Y')\mapsto (Z,Z'):=\left(\int_0^{\cdot}Yd\mathbf{X},\, Y\right),
		\]
		$$\mathscr{D}^{p,q}_{X}([0,T],\mathcal{L}(\Rr^d,\Rr^k{))\rightarrow} \mathscr{D}^{p, q}_{X}([0,T],\mathcal{L}(\Rr^d,\Rr^k{))}
		$$
		is continuous and 
		\begin{align*}
		\|Z,Z'\|_{p, q,X}\leq \|Y\|_{p}+\|Y'\|_{\infty}\|\mathbb{X}\|_{p/2}
		+C\left(\|X\|_{p}\|R^{Y}\|_{q}+\|Y'\|_{p}\|\mathbb{X}\|_{p/2}\right).
		\end{align*}
	\end{theorem} 
	In the theorem $Y'_t\mathbb{X}_{t,s}$ is interpreted via the natural inclusion:
	\begin{align*}
	\mathcal{L}(\Rr^d,\mathcal{L}(\Rr^d,\Rr^k)) \hookrightarrow\mathcal{L}(\Rr^d\otimes\Rr^d,\Rr^k)\\
	Y'(v\otimes w) := \underbrace{Y'(v)}_{\in \mathcal{L}(\Rr^d,\Rr^k)} (w),\, v, w \in \Rr^d
	\end{align*}
	(in coordinates as {$(Y'_t\mathbb{X}_{t,s})_l:=\sum_{i, j=1}^d (Y')^{l j}_i\mathbb{X}^{i j},\, l=1,\ldots, k$}). 
	\begin{proof}
		The proof of the first estimate is similar to \cite[Theorem 31]{FritzShekhar} so  we omit it here.
		The second estimate of the theorem follows from the first one and the triangle inequality by noting {that} ${R^Z(t, s)}=\int_t^s Yd\textbf{X}-Y_tX_{t,s}$, and 
		\begin{align}\label{eq:RZ}
		|{R^Z(t,s)}|\leq \Big|\int_t^s Y\, d\textbf{X}-Y_tX_{t,s}-Y'_t\mathbb{X}_{t,s}\Big|\,+\,|Y'_t\mathbb{X}_{t,s}|\nonumber\\
		\leq
		C\left(\|{R}^Y\|_{q,[t,s]}\|X\|_{p,[t,s]}+\|Y'\|_{p,[t,s]}\|\tilde{\mathbb{X}}\|_{p/2,[t,s]}\right)\
		+\|Y'\|_{\infty}\|\mathbb{X}\|_{p/2,[t,s]}.
		\end{align}
		hence, using that the $q$-th power of the right-hand side is superadditive, we obtain
		\begin{align*}
		\|R^Z\|_{q,[0,T]}\leq C \left(\|{R}^Y\|_{q,[0, T]}\|X\|_{p,[0,T]}+\|Y'\|_{p,[0,T]}\|\tilde{\mathbb{X}}\|_{p/2,[0,T]}\right)
		+\|Y'\|_{\infty}\|\mathbb{X}\|_{p/2, [0,T]}.
		\end{align*}
		{Consequently,}
		\begin{align*}
		\|Z,Z'\|_{p, q, X}=\|Y\|_{p}+\|R^Z\|_{q}
		\leq \|Y\|_{p}+\|Y'\|_{\infty}\|\mathbb{X}\|_{p/2}\\
		+C\left(\|X\|_{p}\|R^{Y}\|_{q}+\|Y'\|_{p}\|\mathbb{X}\|_{p/2}\right).
		\end{align*}
	\end{proof}

	\section{Non-anticipative functionals of rough paths }
	\label{chp:actions}
	In this section, we study the behaviour of rough paths under the actions of regular non-anticipative functionals. We construct a rough integral for integrands given by sufficiently regular non-anticipative functionals. 
	
	\subsection{Non-anticipative functionals}
	Let us recall briefly the definition of non-anticipative functionals and their derivatives \cite{cont2012}.
	A functional
	$F:[0, T] \times D([0, T], \mathbb{R}^d) \rightarrow \mathbb{R}$ on  the space $D([0, T], \mathbb{R}^d)$ of \cadlag paths is called non-anticipative if it satisfies a causality property:
	\begin{align}\label{0eq:Non-AnticipativityCondition}
	F(t, {x}) = F(t, {x}_t) \qquad \forall {x} \in \Omega,
	\end{align}
	where ${x}_t$ represents the path ${x}$ stopped at time $t$. It turns out that it is convenient to define these non-anticipative functionals on the space of stopped paths, where we define a stopped path as an equivalence class in $[0, T] \times D([0, T], \mathbb{R}^d)$ with respect to the following equivalence relation:
	\begin{align*}
	(t, {x}) \sim (t', {x}') \Longleftrightarrow (t = t' \,\,\, \mathrm{and} \,\,\, {x}(t \wedge \cdot) = {x}'(t'\wedge \cdot)).
	\end{align*}
	It is possible to endow this space with a metric structure, via the following distance function:
	\begin{align*}
	d_\infty((t, {x}), (t', {x}')) := \sup_{u \in [0, T]} |{x}(u \wedge t) - {x}'(u \wedge t')| + |t - t'| \\
	= ||{x} - {x}'||_\infty + |t-t'|.
	\end{align*}
	The space $(\Lambda_T^d, d_\infty)$ is then a complete metric space. Now, every map satisfying condition \eqref{0eq:Non-AnticipativityCondition} can be viewed as a functional on the space of stopped paths. 
	\begin{definition}[Non-anticipative functional] A  \textbf{non-anticipa\-tive functional}  is a measurable map $F: (\Lambda_T^d, d_\infty) \rightarrow \mathbb{R}^k$.
		We denote $\mathbb{C}^{0,0}(\Lambda_T^d)$ the set of continuous maps $F: (\Lambda_T^d, d_\infty) \rightarrow \mathbb{R}^k$.
	\end{definition} 
	$F\in \mathbb{C}^{0,0}(\Lambda_T^d)$ implies joint continuity in $(t,{x})$. We will also need some weaker notions of continuity \cite{cont2012}.
	\begin{definition}
		A non-anticipative functional $F$ is said to be:
		\begin{itemize}
			\item \textbf{continuous at fixed times} if for any $t \in [0, T]$, $F(t, \cdot)$ is continuous w.r.t. the uniform norm $||\cdot||_\infty$, i.e., $\forall {x} \in D([0, T], \mathbb{R}^d)$, $\forall \epsilon > 0 $, $\exists \nu > 0$ such that $\forall {x}' \in D([0, T], \mathbb{R}^d)$:
			\begin{align*}
			||{x}_t- {x}'_t||\infty < \nu \Rightarrow |F(t, {x}) - F(t, {x}')| < \epsilon,
			\end{align*}
			\item \textbf{left-continuous} if $\forall (t, {x}) \in \Lambda_T^d$, $\forall \epsilon > 0$, $\exists \nu > 0$ such that $\forall (t', {x}') \in \Lambda_T^d$:
			\begin{align*}
			t' < t \,\,\, \mathrm{and} \,\,\, d_\infty((t, {x}), (t', {x}')) < \nu \Rightarrow |F(t, {x}) - F(t', {x}')| < \epsilon.
			\end{align*}
			We denote the set of left-continuous functionals by $\mathbb{C}^{0,0}_l(\Lambda_T^d)$. 
		\end{itemize}
	\end{definition} 
	We will also need a notion of local boundedness for functionals.
	\begin{definition}
		A functional $F$ is called \textbf{boundedness-preserving} if for every compact subset $K$ of $\mathbb{R}^d$, $\forall t_0 \in [0, T]$, $\exists C(K, t_0) > 0$ such that:
		\begin{align*}
		\forall t \in [0, t_0],\,\,\, \forall (t, {x}) \in \Lambda_T^d: \,\,\, {x}([0, t]) \subset K \Rightarrow |F(t, {x})| < C(K, t_0).
		\end{align*}
		We denote the set of boundedness preserving functionals by $\mathbb{B}(\Lambda_T^d)$.
	\end{definition}
	
	We now recall some definition of differentiability for non-anticipative functionals. Given $e \in \mathbb{R}^d$ and ${x} \in D([0, T], \mathbb{R}^d)$, we define the vertical perturbation ${x}_t^e$ of $(t, {x})$ as the \cadlag path obtained by adding a jump discontinuity to the path ${x}$ at time $t$ and of size $e$, that is:
	\begin{align*}
	{x}_t^e := {x}_t + e\mathbbm{1}_{[t, T]}.
	\end{align*}
	\begin{definition}
		A non-anticipative functional $F$ is said to be:
		\begin{itemize}
			\item \textbf{horizontally differentiable} at $(t, {x}) \in \Lambda_T^d$ if:
			\begin{align}
			\mathcal{D}F(t, {x}) = \lim_{h \downarrow 0} \frac{F(t + h {x}_t) - F(t, {x})}{h}
			\end{align} 
			exists. If $\mathcal{D}F$ exists for all $(t, {x}) \in \Lambda_T^d$, then $\mathcal{D}F$ defines a new non-anticipative functional, called the horizontal derivative of $F$.
			\item \textbf{vertically differentiable} at $(t, {x}) \in \Lambda_T^d$ if the map:
			\begin{align*}
			g_{_{(t, {x})}} : \mathbb{R}^d &\rightarrow \mathbb{R}\\
			e &\mapsto F(t, {x}_t+e\ 1_{[t,T]})
			\end{align*}
			is differentiable at $0$. In that case, the gradient at $0$ is called the Dupire derivative (or vertical derivative) of $F$ at $(t, {x})$:
			\begin{align}
			\nabla_{x} F(t, {x}) = \nabla g_{_{(t, {x})}}(0) \in \mathbb{R}^d,
			\end{align}
			that is, we have $\nabla_{x} F(t, {x}) = (\partial_i F(t, {x}), i = 1, ..., d)$ with 
			\begin{align*}
			\partial_i F(t, {x}) = \lim_{h \rightarrow 0} \frac{F(t, {x}_t + he_i\mathbbm{1}_{[t, T]}) - F(t, {x}_t)}{h},
			\end{align*}
			where $(e_i, i=1, ..., d)$ is the canonical basis of $\mathbb{R}^d$. If $\nabla_{x} F$ exists for all $(t, {x}) \in \Lambda_T^d$, then $\nabla_{x} F : \Lambda_T^d \rightarrow \mathbb{R}^d$ defines a non-anticipative functional called the vertical derivative of $F$.
		\end{itemize}
	\end{definition}
	Note that, since the objects that we obtain when computing these derivatives are still non-anticipative functionals, we can reiterate these operations and introduce higher order derivatives, such as $\nabla_{x}^2 F$. This leads to the definition of the following class of smooth functionals.
	\begin{definition}
		We define $\mathbb{C}^{1,k}_b(\Lambda_T^d)$ as the set of non-anticipative functionals $F: (\Lambda_T^d, d_\infty) \rightarrow \mathbb{R}$ which are:
		\begin{itemize}
			\item horizontally differentiable, with $\mathcal{D}F$ continuous at fixed times;
			\item $k$ times vertically differentiable, with $\nabla_{x}^j F \in \mathbb{C}_l^{0,0}(\Lambda_T^d)$ for $j = 0, \ldots, k$;
			\item $\mathcal{D}F, \nabla_{x} F, \ldots, \nabla_{x}^k F \in \mathbb{B}(\Lambda_T^d)$.
		\end{itemize}
	\end{definition}
	
	Throughout the section, we will work with functionals satisfying the following assumption of Lipschitz continuity in the metric $d_{\infty}:$
	
	\begin{hyp}\label{chp3ass.lipd}
		$\exists K>0, \quad \forall (t, X),\, (t', X')\in \Lambda^d_T,$
		\[    |F(t,X)-F(t',X')|\leq K d_{\infty}((t, X),\, (t', X')).
		\]
		The space of such functionals is denoted by $Lip(\Lambda_T^d, d_{\infty}).$
	\end{hyp}
	
	We note that the above property implies the following Lipschitz continuity property
	
	\begin{hyp}[Uniformly Lipschitz continuity]\label{chp3ass.Fholder}
		$\exists K>0, \quad \forall X,\, X'\in D([0,T], \Rr^d),$
		\[    |F(t,X)-F(t,X')|\leq K \|X_t-X'_t\|_{\infty}.
		\]
		We denote the space of such functionals by $Lip(\Lambda_T^d, \|\cdot\|_{\infty}).$
	\end{hyp}
	\begin{hyp}[Horizontal Lipschitz continuity]\label{chp3ass.horizlip}
		$\exists K>0, \quad \forall X\in D([0,T], \Rr^d),$
		\[    |F(t,X)-F(t',X)|\leq K |t'-t|.
		\]
		We denote the space of such functionals by $hLip(\Lambda_T^d).$
	\end{hyp}It is not hard to see that $Lip(\Lambda_T^d, d_{\infty})=Lip(\Lambda_T^d, \|\cdot\|_{\infty}) \cap hLip(\Lambda_T^d)$
	
	\subsection{Actions of functionals on rough paths}
	We are now ready to study actions of regular non-anticipative functionals on rough paths and controlled paths.
	
	The following lemma is a particular case of \cite{Anna}[Lemma 5.11], which allows to approximate paths with finite variation by piece-wise affine paths. 
	
	\begin{lemma}\label{chp2ilem:approxomega}
		For any path $  X\in C^{p-var}([0,T],\Rr^d)$ and an integer $N>1$ and interval $[t,s]$ there exists  $ X^N\in C([0,s]; \Rr^d)$ such that it is a piece-wise linear on $[t,s]$
		\begin{itemize}
			\item $  X^N$ coincides with $  X$ on $[0,t]:$     
			\[
			X^N_t=  X_t,
			\]
			\item $X\to  X^N$ is a linear map and
			\[
			\|X^N\|_{\infty}\leq \|X\|_{\infty},\, \|X^N\|_{p, [t,s]}\leq \|X\|_{p, [t,s]},
			\]
			\item $  X^N$  approximates $X$:
			\[
			\|  X-  X^N\|_{\infty}\leq C N^{-\nu}\|X\|_{p, [t,s]},
			\]
			
			\item $ X^N$ has a bounded variation on $[t,s]$ with the variation
			\[
			V_1(X^N, t, s):= \int_t^s |d  X^N|\leq C N^{1-\nu}\|X\|_{p, [t,s]},
			\]
			where $\nu:= p^{-1}.$
		\end{itemize}
	\end{lemma}
	
	Next, we recall a corollary of \cite{Anna}[Theorem 5.12], which provides a connection between regular functionals and {controlled path}s. 
	
	\begin{lemma} \label{chp3theorem:controll} Let ${p>2},\,  X\in C^{p-var}([0,T],\mathbb{R}^d)$ and $F\in \mathbb{C}^{0,1}_b(\Lambda^d_T,\mathbb{R}^n)$  with $F$ and  $\nabla_{{x}} F$ in $Lip(\Lambda^d_T, d_{\infty})$. 
		Define
		\begin{equation}\label{chp2ieq:controlledpath}
		\mathcal{R}^{F}_{t,s}(X):=F(s,  X_s)-F(t,  X_t)-\nabla_{{x}}F(t,  X_{t})(X(s)-  X(t)).
		\end{equation}    
		Then there exists a constant $C_{F,T}$ increasing in $T$, which depends on the regularity properties of $F$ and its derivatives locally in a neighbourhood  of  $X$, such that ${\mathcal{R}}^{F}(X)$ has bounded $q_p:=\frac {p^2} {p+1}$-variation and
		\[
		\|{\mathcal{R}}^{F}(X)\|_{q_p, [t,s]}\leq C_{p, F, T}\left(|s-t|+\|  X\|_{p, [t,s]}^{1+p^{-1}}\right).
		\]
		Thus the pair $(F(\cdot,   X), \nabla_{{x}} F(\cdot,   X))$ is a controlled   path: $$(F(\cdot,   X), \nabla_{{x}} F(\cdot,   X))\in \mathcal{D}^{p, q_p}_{_X }([0,T], {\Rr^n}),\qquad q_p:=\frac {p^2} {p+1}.$$
	\end{lemma}
	We omit the proof of this lemma as it is based on the same idea as the proof of the next result.
	
	For our purposes, we would like to have a stability result for the estimate in the previous theorem in terms of the underlying path $Y.$ The following result allows us to control the error term $\mathcal{R}^{F}(Y)_{s,t}$ in $Y$ and will be useful in the proof of the existence of solutions to path-dependent RDEs. 
	
	\begin{lemma}[Continuity of Control]\label{chp3theorem:contrcont} Let   $Y_1, Y_2\in C^{p-var}([0,T],\mathbb{R}^d)$ for some $p\in [2, 3)$. Let $F\colon \Lambda^d_T\to V$ be a non-anticipative functional with values in a finite dimensional real vector space $V$. Assume $ F\in \mathbb{C}^{1,2}_b(\Lambda^d_T),\,  \nabla F\in \mathbb{C}^{1,1}_b(\Lambda^d_T)$ and $F, \mathcal{ D} F,\, \mathcal{D}\nabla  F,\, \nabla^2_{X} F\in Lip(\Lambda^d_T, \|\cdot\|_{\infty})$. Furthermore,  if $R, M>0$ are such that 
		\[
		\|Y_1\|^p_{p, [0,T]}\,+\,\|Y_2\|^p_{p, [0,T]}\,+\,T\leq R.
		\]
		and 
		$$\|Y_1-Y_2\|_{\infty, [0, T]}+\|Y_1-Y_2\|_{p, [0,T]}\leq M,$$
		then for all $t\leq  s\in[0,T]$
		\[
		\|\mathcal{R}^{F}(Y_1)-\mathcal{R}^{F}(Y_2)\|_{q_{_p},[t,s]}\leq C_{F, M, R}\left(\|Y_1-Y_2\|^{\nu}_{\infty, [0, s]}+\|Y_1-Y_2\|_{p, [t,s]}^{\nu}\right),
		\]
		where $q_p:=\frac {p^2} {p+1},\, \nu:=p^{-1}$.
		
	\end{lemma}
	\begin{proof}
		We will prove only the case when the values of $F$ are scalar, i.e. $V=\Rr$, for the general case it is enough to use the result for each coordinate of $F$. Let $\omega$ be a superadditive map on intervals of $[0,T]$, given by
		\[
		\omega([u,v]):=\|Y_1\|^p_{p, [u,v]}\,+\,\|Y_2\|^p_{p, [u,v]}\,+\,|u-v| .
		\]
		
		We start by recalling the following result:
		\begin{lemma*} [see \cite{cont2012}, Proposition 5.26]
			Assume $G\in C^{1,1}_b(\mathcal{W}^d_T)$ and $\lambda$ is a continuous path with finite variation on $[t,s]$, then
			\begin{equation*}
			G(s,\lambda_s)-G(t,\lambda_t)=\int_t^s\mathcal{D} G(u,\lambda_u)du+\int_t^s\nabla  G(u,\lambda_u)d\lambda(u)\tag{*}
			\end{equation*}
			where the second integration is in the Riemann-Stieltjes sense.
			
		\end{lemma*}
		
		Let us fix a  Lipschitz continuous  path $Y;[0,T]\to\Rr^d$, using the above lemma repeatedly for $G=\nabla  ^j F,\,j=0,1,2$, we will obtain a expression of $\mathcal{R}^{F}_{t,s}(Y)$ in terms of the derivatives $\mathcal{D}^i\nabla  ^j F$.  For the sake of convenience we denote by $\partial_i F$ and $Y^i$ respectively the $i$-th coordinates of $\nabla  F$ and $Y.$ We also use Einstein's convention of summation in repeated indexes. Using (*) for $G=F$, we have
		
		\begin{equation}\label{chp3eq:rfy1}
		\mathcal{R}^{F}_{t,s}(Y)=\int_t^s\mathcal{D}F(u,Y_u)du+\int_t^s\left(\partial_i F(u,Y_u)-\partial_i F(t, Y_t)\right)dY^i(u)
		\end{equation}
		For the second term  on the right-hand side of the above identity we use the Lemma (*) with $G=\partial^2_{i} F$  and then Fubini's theorem to get
		\eq{\label{chp3eq:rfy2}
			\int_t^s\left(\partial_i F(u,Y_u)-\partial_i F(t, Y_t)\right)dY^i(u)=\int_t^s\int_t^u \mathcal{D}\partial_i F(r,Y_r)drdY^i(u)\nonumber\\
			+ \int_t^s\int_t^u \partial^2_{ij} F(r,Y_r)dY^j(r)dY^i(u)\nonumber\\
			=\int_t^s\mathcal{D}\partial_i F(r,Y_r)(Y^i(s)-Y^i(r))dr\nonumber\\
			+\int_t^s\partial^2_{ij} F(r,Y_r)(Y^i(s)-Y^i(r))\dot Y^j(r)dr.
		}
		
		Combining \ref{chp3eq:rfy1}, \ref{chp3eq:rfy2}  we arrive to the formula
		\eq{\label{chp3eq:rfyfinal}
			\mathcal{R}^{F}_{t,s}(Y)=\underbrace{\int_t^s\mathcal{D}F(u,Y_u)du}_{:=I_1(Y)}+\underbrace{\int_t^s\mathcal{D}\partial_i F(r,Y_r)Y^i_{r,s}dr}_{:=I_2(Y)}
			+\underbrace{\int_t^s\partial^2_{ij} F(r,Y_r)Y^i_{r, s}\dot Y^j(r)dr}_{:=I_3(Y)}.}
		Let $Y^N$ be the piece-wise linear (affine) approximation of $Y$ given by Lemma \ref{chp2ilem:approxomega}.
		Using the estimates of that lemma, Lipschitz continuity of  
		$\mathcal{D}F, \mathcal{D}\nabla  F,\nabla  ^2 F$  and the following consequences of triangle inequality
		\[
		|a_1b_1-a_2b_2|\leq |a_1-a_2||b_1|+|a_2||b_1-b_2|,
		\]
		\[
		|a_1b_1c_1-a_2b_2c_2|\leq |a_1-a_2||b_1||c_1|+|a_2||b_1-b_2||c_1|+|a_2||b_2||c_1-c_2|,
		\]
		we obtain, for each term of representation \eqref{chp3eq:rfyfinal}, we have
		\[
		|I_1(Y_1^N)-I_1(Y^N_2)|\leq C_F\|Y_1-Y_2\|_{\infty}|s-t|,
		\]
		\begin{align*}
		|I_2(Y_1^N)-I_2(Y^N_2)|\leq C_F\|Y_1-Y_2\|_{\infty}\|Y_1\|_{p, [t,s]}|s-t|\\
		+\|\mathcal{ D}\nabla  F(\cdot, Y^N_2)\|_{\infty}\|Y_1-Y_2\|_{p, [t,s]}|s-t|\\\leq 
		C_{F}\left(\|Y_1-Y_2\|_{\infty}  \omega([t,s])^{\nu}\, +\,\|Y_1-Y_2\|_{p,[t,s]}\right) \omega([t,s]),
		\end{align*}
		and
		\begin{align*}
		|I_3(Y_1^N)-I_3(Y^N_2)|\leq C_F\|Y_1-Y_2\|_{\infty} N^{1-\nu}\|Y_1\|^2_{p, [t,s]}\nonumber\\
		+\|\nabla  ^2 F(\cdot, Y^N_2)\|_{\infty} \|Y_1-Y_2\|_{p,[t,s]} (N^{1-\nu}\|Y_1\|_{p, [t,s]})\nonumber\\
		+\|\nabla  ^2 F(\cdot, Y^N_2)\|_{\infty} \|Y_2\|_{p, [t,s]} (N^{1-\nu}\|Y_1-Y_2\|_{p,[t,s]})\\\leq
		N^{1-\nu}C_{F,R}\left(\|Y_1-Y_2\|_{\infty}  \omega([t,s])^{\nu}\, +\,\|Y_1-Y_2\|_{p,[t,s]}\right) \omega([t,s])^{\nu}.
		\end{align*}
		From these
		\eq{
			\label{chp3eq:rfyholder}
			|\mathcal{R}^{F}_{t,s}(Y^N_1)-\mathcal{R}^{F}_{t,s}(Y^N_2)|
			\leq C_{F, R} (\|Y_1-Y_2\|_{\infty} + \|Y_1-Y_2\|_{p,[t,s]}) \omega([t,s])\nonumber\\
			+\,C_F N^{1-\nu} \left(\|Y_1-Y_2\|_{\infty}  \omega([t,s])^{\nu}\, +\,\|Y_1-Y_2\|_{p,[t,s]}\right) \omega([t,s])^{\nu}\nonumber\\
			\leq C_{F, T, R}(\|Y_1-Y_2\|_{\infty} \omega([t,s])^{\nu}+\|Y_1-Y_2\|_{p, [t,s]})\nonumber\\
			\times( \omega([t,s])^{1-\nu}+N^{1-\nu}  \omega([t,s])^{\nu}).
		}
		
		On the other hand since $(Y^N)_t=Y_t$ and $Y^N(s)=Y(s)$ we can replace $\mathcal{R}^{F}_{t,s}(Y^N_1)$ and $\mathcal{R}^{F}_{t,s}(Y^N_2)$ respectively with $\mathcal{R}^{F}_{t,s}(Y_1)$ and $\mathcal{R}^{F}_{t,s}(Y_2)$ 
		with an error
		\eq{\label{chp3eq:rfyapprox}
			|\mathcal{R}^{F}_{t,s}(Y^N_i)-\mathcal{R}^{F}_{t,s}(Y_i)|=|F(s,Y^N_{i})-F(s,Y_{i})|\nonumber\\
			\leq C_F\|Y^N_{i}-Y_{i}\|_{\infty}\leq C_F\|Y_{i}\|_{p, [t,s]}N^{-\nu}\leq C_FN^{-\nu} \omega([t,s])^{\nu} ,\, i=1,\, 2.
		}
		Let
		\[
		d_{\nu}(Y_1,Y_2):= \|Y_1-Y_2\|_{\infty} \omega([t,s])^{\nu}+\|Y_1-Y_2\|_{p, [t,s]},
		\]
		from \eqref{chp3eq:rfyholder} and \eqref{chp3eq:rfyapprox} and triangle inequality
		\begin{align*}
		|\mathcal{R}^{F}_{t,s}(Y_1)-\mathcal{R}^{F}_{t,s}(Y_2)|\leq C_{F,R}
		\big(d_{\nu}(Y_1,Y_2) \omega([t,s])^{1-\nu}\\
		+N^{1-\nu}d_{\nu}(Y_1,Y_2)  \omega([t,s])^{2\nu}+N^{-\nu} \omega([t,s])^{\nu}\big).
		\end{align*}
		To optimize the above bound, we choose $N$ so that
		$$
		N^{1-\nu} d_{\nu}(Y_1,Y_2)  \omega([t,s])^{\nu}\approx N^{-\nu} \omega([t,s])^{\nu}
		$$
		i.e.
		$
		N\approx d_{\nu}(Y_1,Y_2)^{-1}.
		$
		Hence
		\begin{align*}
		|\mathcal{R}^{F}_{t,s}(Y_1)-\mathcal{R}^{F}_{t,s}(Y_2)|\leq C_{F, R}\, 
		\left(d_{\nu}(Y_1,Y_2) \omega([t,s])^{1-\nu}  +  d_{\nu}(Y_1,Y_2)^{\nu} \omega([t,s])^{\nu}\right)\\
		\leq C_{F, M, R}\left(\|Y_1-Y_2\|^{\nu}_{\infty} \omega([t,s])^{\nu+\nu^2}+\|Y_1-Y_2\|_{p, [t,s]}^{\nu} \omega([t,s])^{\nu} \right).
		\end{align*}
		Using the inequality $(|a|+|b|)^q\leq 2^{q}(|a|^q+|b|^q), \forall q>0$, we have
		\begin{align*}
		|\mathcal{R}^{F}_{t,s}(Y_1)-\mathcal{R}^{F}_{t,s}(Y_2)|^{q_p}
		\leq C_{F, M, R}\left(\|Y_1-Y_2\|^{p/(p+1)}_{\infty} \omega([t,s])+\|Y_1-Y_2\|_{p, [t,s]}^{p/(p+1)} \omega([t,s])^{p/(p+1)} \right).
		\end{align*}
		It remains to note that the right-hand side is superadditive function of the interval $[t,s]$, thus summing up such inequalites over a partitions of $[t,s]$, yields
		\[
		\|\mathcal{R}^{F}_{t,s}(Y_1)-\mathcal{R}^{F}_{t,s}(Y_2)\|_{q_p, [t,s]}\leq C_{F, M, R}\left(\|Y_1-Y_2\|^{\nu}_{\infty} \omega([t,s])^{\nu+\nu^2}+\|Y_1-Y_2\|_{p, [t,s]}^{\nu} \omega([t,s])^{\nu} \right),
		\]
		hence the result.
	\end{proof}
	
	As a consequence of the previous theorem, we can control the $p$-variation distance of the images of two paths under a regular functional:
	
	\begin{corollary}\label{chp3Cor:controlcont}Let   $Y_1, Y_2\in C^{p-var}([0,T],\mathbb{R}^d)$ for some $p\in [2,3)$ and  $F\colon \Lambda^d_T\to V$. Assume $ F\in \mathbb{C}^{1,2}_b(\Lambda^d_T),\,  \nabla F\in \mathbb{C}^{1,1}_b(\Lambda^d_T)$ and $\mathcal{ D} F,\, \mathcal{D}\nabla  F,\, \nabla  F,\, \nabla^2_{X} F$ are in $Lip(\Lambda^d_T, \|\cdot\|_{\infty})$.  Then 
		\[
		\|F(\cdot,Y_1)-F(\cdot, Y_2)\|_{p, [t,s]}\leq C_{F, M, R}\left(\|Y_1-Y_2\|_{\infty, [0,s]}^{\nu}\,+\,\|Y_1-Y_2\|_{p, [t,s]}^{\nu}\right),
		\]
		provided
		\[
		|Y_1(0)|+\|Y_1\|_{p, [0,T]},\,|Y_2(0)|+\|Y_2\|_{p, [0,T]}\leq R,
		\]
		and
		\[
		\|Y_1-Y_2\|_{\infty, [0,T]}+\|Y_1-Y_2\|_{p, [0,T]}\leq M.
		\]
	\end{corollary}
	
	\begin{proof}
		This follows immediately from Lemma \ref{chp3theorem:contrcont} and following identity
		\begin{align*}
		(F(\cdot,Y_1)-F(\cdot, Y_2))_{t,s}=(\nabla  F(t,Y_1)(Y_1)_{t, s}-\nabla  F(t, Y_2)(Y_2)_{t, s})\\
		+\mathcal{R}^{F}_{t,s}(Y_1)-\mathcal{R}^{F}_{t,s}(Y_2).
		\end{align*}
		Indeed, we get
		\begin{align*}
		|(F(\cdot,Y_1)-F(\cdot, Y_2))_{t,s}|\leq
		|(\nabla  F(t,Y_1)-\nabla  F(t, Y_2))(Y_1)_{t, s}|\,\\
		+\,|\nabla  F(t, Y_2)(Y_1-Y_2)_{t, s})|\,
		+\,|\mathcal{R}^{F}_{t,s}(Y_1)-\mathcal{R}^{F}_{t,s}(Y_2)|\\ \leq C_F \|Y_1-Y_2\|_{\infty} \|Y_1\|_{p, [t,s]}\,+\, C_F \|Y_1-Y_2\|_{p, [t,s]}    +\,|\mathcal{R}^{F}_{t,s}(Y_1)-\mathcal{R}^{F}_{t,s}(Y_2)|.
		\end{align*}
		From this and the Lemma \ref{chp3theorem:contrcont}
		\begin{align*}
		\|F(\cdot,Y_1)-F(\cdot, Y_2)\|_{p,[t,s]}
		\leq C_F \|Y_1-Y_2\|_{\infty} \|Y_1\|_{p, [t,s]}\,+\, C_F \|Y_1-Y_2\|_{p, [t,s]}\\
		+ C_{F, M, R} \Big(\|Y_1-Y_2\|_{\infty}^{\nu}+\|Y_1-Y_2\|_{p, [t,s]}^{\nu}\Big)
		\end{align*}

	\end{proof}

	We will now use Lemma \ref{chp3theorem:controll} to define rough integrals with regular non-anticipative integrands:
	\begin{theorem}[Rough integral for functionals]\label{chp3theorem:roughpathint} Let $\bm{X}:=(X, \mathbb{X})\in \mathcal{C}^{p-var}([0,T],\mathbb{R}^d)$ be a rough path for some $p\in [2, \sqrt{2}+1)$.
		Assume $F\in \mathbb{C}^{0,1}_b(\Lambda^d_T,\mathbb{R}^n)$  with $F$ and  $ \nabla  F$ are  locally horizontally Lipschitz continuous and are in  $Lip(\Lambda^d_T, \|\cdot\|_{\infty})$.
		Then the rough integral
		\begin{equation}\label{chp3def.integral}
		\int_0^u F(s, X)d\bm{X}(s)
		\end{equation}
		exists. Moreover,
		\begin{align*}
		&\left|\int_t^s Fd\bm{X}-F(t,X)(X(s)-X(t))-\nabla F(t,X)\mathbb{X}_{t,s}\right|\lesssim\\
		&
		\left(\|X\|_{p,[t,s]}\|\mathcal{R}^{F}(X)\|_{q_{p},[t,s]}+\|\nabla  F(\cdot,X_{\cdot})\|_{p,[t,s]}\|\mathbb{X}\|_{p/2,[t,s]}\right).
		\end{align*}
		
	\end{theorem}

	\begin{proof}
		By Lemma \ref{chp3theorem:controll} the  $(F(\cdot, X), \nabla  F(\cdot, X))\in \mathscr{D}^{p,q_p}_{X}([0,T],\Rr)$ (is controlled by $X$ in the sense of Definition \ref{chp02def:conrtollab}).  Thus the result follows by Theorem \ref{chp02theorem:itegralcontrol}, one only needs to check that

		for $p<\sqrt{2}+1$, we have $p^{-1}+q_p^{-1}= \frac{2p+1}{p^2}>1$.
	\end{proof}

	We continue to investigate the actions of regular functionals on controlled  paths. The next result asserts the invariance of controlled paths under the action of regular functionals: 
	
	\begin{theorem}\label{chp3theorem:actiononpaths} Let $X\in C^{p-var}([0,T],\mathbb{R}^d)$ and $(Y,Y')\in \mathscr{D}^{p,q_p}_{X}([0,T],\Rr^k),$ where $q_p=p^2/(p+1)$ for $p\in [2, \sqrt{2}+1)$.
		Let $F\colon \Lambda^{d\times k}_T\to \mathcal{L}(\Rr^d,\Rr^m)$ be a non-anticipative functional. Assume $\nabla  F\in \mathbb{C}^{1,2}_b(\Lambda^{d\times k}_T), \nabla F\in \mathbb{C}^{1,1}_b(\Lambda^{d\times k}_T)$ and $\mathcal{ D} F,\, \mathcal{D}\nabla  F,\, \nabla  F,\, \nabla^2_{X} F$ are in $Lip(\Lambda^{d\times k}_T, \|\cdot\|_{\infty})$. 
		Then
		\[
		(F(\cdot,Y), F(\cdot,Y)'):=(F(\cdot,Y),\nabla  F(\cdot,Y)Y')\in \mathscr{D}^{p,q_p}_{X}([0,T],\Rr^m).
		\]
		Furthermore, assuming $1+|Y'_0|+\|Y,Y'\|_{p,q_p,X}\leq M$, we have
		\[
		\|F(\cdot, Y)'\|_{p,[t,s]}\leq C_{F,M}(|s-t|+\|Y\|_{p,[t,s]}\, +\, \|Y'\|_{p, [t,s]}).
		\]
		and
		\[
		\|R^{F(\cdot, Y)}\|_{q_p, [t,s]}
		\leq C_{F,T}(
		\|R^Y\|_{q_p, [t,s]}+ \|Y\|_{p, [t,s]}^{1+\nu}+|s-t|).
		\]

	\end{theorem}
	
	\begin{proof}
		From Lipschitz continuity Assumptions on $F$ we obtain
		\begin{align*}
		|F(s,Y_s)-F(t,Y_t)|\leq |F(s,Y_s)-F(s,Y_{t})|+|F(s,Y_{t})-F(t,Y_t)|\\
		\leq C\|Y_s-Y_{t}\|_{\infty}+C|s-t|\leq C(\|Y\|_{p, [t,s]}+|s-t|),
		\end{align*}
		hence
		\eq{\label{chp3eq:DYp}
			\|F(\cdot, Y)\|_{p,[t,s]}\leq c_F(\|Y\|_{p,[t,s]}+|s-t|),
		}
		similarly 
		\eq{\label{chp3eq:DFYnu}
			\|\nabla F(\cdot, Y)\|_{p,[t,s]}\leq c_F(\|Y\|_{p,[t,s]}+|s-t|).
		}
		From the last inequality and triangle inequality, we get
		\begin{align*}
		|(F(\cdot,Y)')_{t,s}|= |(\nabla F(\cdot,Y) Y'(\cdot) )_{t, s}|\\
		\leq
		|\nabla F(t,Y)(Y')_{t, s}|\,
		+\,|(\nabla F(\cdot, Y))_{t, s}Y'(s)|\\
		\leq\Big( \|\nabla F(\cdot,Y)\|_{\infty}\|Y'\|_{p,[t,s]}\,+\,\|\nabla  F(\cdot, Y)\|_{p,[t,s]}\|Y'\|_{\infty, [t,s]}\Big).
		\end{align*}
		Plugging in \eqref{chp3eq:DFYnu}
		\eq{\label{chp3eq:FYnu}
			\|F(\cdot, Y)'\|_{p,[t,s]}\leq \|\nabla F(\cdot,Y)\|_{\infty}\|Y'\|_{p,[t,s]}+C_{F,M}\|Y'\|_{\infty}(|s-t|+ \|Y\|_{p, [t,s]}).
		}
		which with $\|Y'\|_{\infty, [0, T]}\leq |Y'(0)|\,+\, \|Y'\|_{p,[0,T]}\leq M$ implies the first inequality
		Next, for $R^F\equiv R^{F(\cdot, Y)}$, we have
		
		\begin{align}\label{chp3eq:RFts}
		R^F_{t,s}=F(s,Y_s)-F(t,Y_t)-\nabla  F(t,Y_t)Y'_tX_{t,s}=\nonumber\\
		F(s,Y_s)-F(t,Y_t)-\nabla  F(t,Y_t)Y_{t,s}+\nabla  F(t,Y_t)R^Y_{t,s}\nonumber\\
		=\mathcal{R}^{F}_{t,s}(Y)+\nabla  F(t,Y_t)R^Y_{t,s}.
		\end{align}
		To estimate the above, note that from Lemma \ref{chp3theorem:controll}
		\[
		\|\mathcal{R}^{F}(Y)\|_{q_p, [t,s]}\leq C_{F, T} \left[ |s-t|+\|Y\|_{p,[t,s]}^{1+\nu}\right]
		\]
		thus \eqref{chp3eq:RFts} yields
		\begin{align}\label{chp3eq:RFfin}
		\|R^F\|_{q_p, [t,s]}\leq \|\mathcal{R}^{F}(Y)\|_{q_p, [t,s]}+\|\nabla  F(\cdot ,Y)\|_{\infty}\|R^Y\|_{q_p, [t,s]}\nonumber\\
		\leq C_{F,M, T}(
		\|R^Y\|_{q_p, [t,s]}+ \|Y\|_{p, [t,s]}^{1+\nu}+|s-t|).
		\end{align}
		
	\end{proof}

	\section{Path-dependent Differential Equations driven by rough paths}
	\label{chp:existence}
	
	\subsection{The setting of the problem}
	We now turn to our main objective: the study of path-dependent rough differential equations (RDEs).
	Let $(X,\mathbb{X})\in  \mathcal{C}^{p-var}([0,T],\mathbb{R}^d)$ be a given  rough path. We are interested in the following differential equation
	\begin{equation}\label{chp3ode}
	\begin{cases}
	dY(s)= b(s, Y_s)ds+\sigma(s,Y_s)dX (s),\\
	Y_{t_0}=\xi_{t_0},
	\end{cases}
	\end{equation}
	where
	$
	b\colon \Lambda_T^k\to\Rr^k$ and $\sigma\colon \Lambda_T^k \to \mathcal{L}(\Rr^d,\Rr^k)$
	are non-anticipative functionals. Here $\mathcal{L}(V,W)$ denotes the set of linear operators between linear spaces $V, W$, by a slight abuse of notation, we identify $\mathcal{L}(\Rr^d,\Rr^k)$ with the space of $d\times k$ matrices and the Euclidean space $\Rr^d\otimes\Rr^k\equiv\Rr^{d\times k}.$
	
	To  define solutions to this equation, we assume $\sigma$ satisfies the conditions of Theorem \ref{chp3theorem:actiononpaths}. Then
	\begin{align*}
	(Y, Y')\in\mathscr{D}^{p,q_p}_{X } ([0,T],\Rr^k)
	\implies(\sigma(s,Y_s), \nabla  \sigma(s,Y_s)Y'_s )\in \mathscr{D}^{p,q_p}_{X }([0,T],\mathcal{L}(\Rr^d,\Rr^k))
	\end{align*}
	and the equation \ref{chp3ode} may be understood as  a rough integral equation:
	\[
	Y(t)=\xi_{t_0}\,+\, \int_{t_0}^t b(s, Y)\, \d s\,+\,\int_{t_0}^t \sigma(s, Y) d\bm{X}
	\]
	where $\int_{t_0}^t \sigma(s, Y) d\bm{X}$ is the rough integral of the controlled  path $$\left(\Xi, \Xi'\right):=(\sigma(s,Y), \nabla  \sigma(s,Y)Y'(s)).$$ More precisely, we have the following definition
	\begin{definition}\label{chp3def.odesol}
		Let $\bm{X}=(X, \mathbb{X})\in \mathcal{C}^{p-var}([0,T],\Rr^d)$ be rough path over $X$, with $p\in [2, \sqrt{2}+1)$. Assume  $b\in \mathbb{C}^{0,0}_b(\Lambda^k_T,\mathbb{R}^k),\, \sigma\in \mathbb{C}^{0,1}_b(\Lambda^k_T,\mathbb{R}^{d\times k})$  with $\sigma$ and  $ \nabla \sigma$ are  locally horizontally Lipschitz continuous and are in  $Lip(\Lambda^d_T, \|\cdot\|_{\infty})$. A controlled  path $(Y, Y')\in\mathscr{D}^{p,q_p}_{X }([0,T],\Rr^k)$ is called a solution to \ref{chp3ode} if
		\begin{equation}\label{chp3eq:odedef}
		Y(t)=\xi_{t_0}\,+\, \int_{t_0}^t b(s, Y)\, \d s\,+\,\int_{t_0}^t \sigma(s, Y) d\bm{X},
		\end{equation}
		where the second integral is understood as the rough integral for the controlled path $$s\in[0,T]\mapsto (\sigma(s,Y), \nabla  \sigma(s,Y)Y'(s))$$ (which exists due to Theorem \ref{chp3theorem:roughpathint}).
	\end{definition}
	
	Next we specify the assumptions on the coefficients in terms of regularity in Dupire's sense  \cite{CF09,D09}.
	\begin{hyp}\label{ass:b}
		The functional $b\colon \Lambda_T^k\to\Rr^k$ is Lipschitz continuous in $d_{\infty}$; $b\in Lip(\Lambda^{d}_T, d_{\infty})$
	\end{hyp}
	
	\begin{hyp}\label{ass:sigma}
		For the vector field $\sigma\colon \Lambda_T^k \to \mathcal{L}(\Rr^d,\Rr^k)$, we assume
		\begin{itemize}
			\item $\sigma\in \mathbb{C}^{1,2}_b(\Lambda^{d}_T, \mathcal{L}(\Rr^d,\Rr^k)), \nabla \sigma\in \mathbb{C}^{1,1}_b(\Lambda^{d}_T, \mathcal{L}(\Rr^k,\mathcal{L}(\Rr^d,\Rr^k)))$
			\item The derivatives $\sigma,\, \mathcal{ D} \sigma,\, \nabla  \sigma,\,\mathcal{D}\nabla  \sigma,\,  \nabla^2 \sigma$ are Lipschitz continuous in $d_{\infty}$.
		\end{itemize}
	\end{hyp}
			The pioneering work of B. Dupire \cite{D09} and the works by R. Cont and D.A. Fourni\'er \cite{CF09}, \cite{CF10B}, \cite{ContFournie2013} have a number of examples of regular functional in the sense of Dupire derivatives. Some further examples are discussed in \cite{ContKalin} and \cite{Kalin}. Here we modify some of these examples to present functionals which satisfy the above assumptions.
	
	\begin{example}\,

		\begin{enumerate}
			\item {\bf {Running Maximum}:} One of the basic examples of path-dependent functionals is the running maximum. Let $z\colon[0,T]\to \Rr_+$ define
			$$m(t, z):=\max_{s\in[0,t]}z(s)$$
			One can easily check that $m\in Lip(\Lambda^{1}_T, d_{\infty})$, moreover $m$ is boundedness preserving and horizontally differentiable with $\mathcal{ D}m=0.$ However, in general this functional may fail to be vertically differentiable at the point of the maximum of $z$. Following Dupire \cite{D09} we consider the following approximation of the running maximum
			\[
			M_{\epsi, h}(t, z):=
			\begin{cases}
			m(t, z)-\epsi,\, &0\leq z(t)\leq m(t, z)-2\epsi,\\
			m(t, z)-\epsi\,+\,h(z(t)-(m(t, z)-2\epsi)),\, &m(t, z)-2\epsi\leq z(t)\leq m(t, z),\\
			z(t),\, &z(t)\geq m(t, z).
			\end{cases}
			\]
			As shown in \cite{D09} for $h(z)=z^2/(4\epsi)$ the functional  $M_{\epsi, h}$ is twice vertically differentiable. More generally, if we take $h$ to be $C^2$ function with $$h(0)=h'(0)=h''(0)=0\text{ and }h(2\epsi)=\epsi, h'(2\epsi)=1, h''(2\epsi)=0$$ then $M_{\epsi, h}$ is $\mathbb{C}^{1,2}$ functional with
			\[
			\nabla_x M_{\epsi, h}(t, z):=
			\begin{cases}
			0\, &0\leq z(t)\leq m(t, z)-2\epsi,\\
			h'(z(t)-(m(t, z)-2\epsi)),\, &m(t, z)-2\epsi\leq z(t)\leq m(t, z),\\
			1,\, &z(t)\geq m(t, z),
			\end{cases}
			\]
			and
			\[
			\nabla^2_x M_{\epsi, h}(t, z):=
			\begin{cases}
			0\, &0\leq z(t)\leq m(t, z)-2\epsi,\\
			h''(z(t)-(m(t, z)-2\epsi)),\, &m(t, z)-2\epsi\leq z(t)\leq m(t, z),\\
			0,\, &z(t)\geq m(t, z).
			\end{cases}
			\]
			Furthermore, if $h''$ is Lipschitz continuous $M_{\epsi, h}$ satisfies the conditions of Assumption \ref{ass:sigma}, note however that the functional $M_{\epsi, h}$ is not Fr\'echet differentiable.
			
			The above functionals can be adapted for multidimensional paths. Let $\varphi\colon\Rr^d\to\Rr_+$ be a Lipschitz continuous functional, consider the following non-anticipative functional:
			\[
			b(t, x):=m(t,\,\varphi\circ x ).
			\]
			Under the assumptions on $\varphi$ one can easily check that  $b$ is boundedness preserving and $b\in Lip(\Lambda^{d}_T, d_{\infty})$. If the function $\varphi\in C^2$  with Lipschitz continuous derivatives then the functional 
			\[
			\sigma(t,x):=M_{\epsi, h}(t,\,\varphi\circ x )
			\]
			satisfies Assumption \ref{ass:sigma}.
			
			\item{\bf {Discrete time dependence}:} Let  $t_1<\ldots<t_m$ be given time-points in $[0,T]$ and let $\phi\colon[0,T]\times (\Rr^k)^m\to \Rr^N$ be a Lipschitz continuous function. Define a functional $\sigma\Lambda_T^k\to\Rr^N$ as follows
			\[
			\sigma(t, x):=\phi(t,\, x(t\wedge t_1),\,\ldots,\, x(t_m\wedge t)).
			\]
			Furthermore, if $\phi\in C^{1,2}(\colon[0,T]\times (\Rr^k)^m, \,\Rr^N)$  and $\nabla\phi\in C^{1,1}$ with Lipschitz continuous derivatives, then $\sigma$ satisfies the regularity properties of Assumption \ref{ass:sigma}. Indeed, it follows from the following formula for the vertical derivatives
			\[
			\nabla_{{x}}\sigma(t, x)=\sum_{i\colon t_i\geq t}\nabla_{x_i}\phi(t,\, x(t\wedge t_1),\,\ldots,\, x(t_m\wedge t)),
			\]
			and 
			\[
			\nabla^2_{x}\sigma(t, x)=\sum_{i, j\colon t_i, t_j\geq t}\nabla^2_{x_i x_j}\phi(t,\, x(t\wedge t_1),\,\ldots,\, x(t_m\wedge t)).
			\]
			\item{\bf {Integral dependence}:} Let $\psi\colon[0,T]\times D([0,T], \Rr^d)\times \Rr^d\to \Rr^N$ be a Lipschitz continuous functional, then
			\[
			F(t, x)=\int_0^t \psi(s, x_s, x(t)) ds,
			\]
			is in $ Lip(\Lambda^{d}_T, d_{\infty})$ and is horizontally differentiable with $\mathcal{D}F(t, x)=\psi(t, x_t, x(t))$. If furthermore $\phi$ is twice differentiable in the last variable then $F$ is of class $\mathbb{C}^{1,2}$ with the corresponding derivates
			\[
			\nabla_{{x}}F(t, x)=\int_0^t \nabla_{y}\psi(s, x_s, x(t)) ds,\,\,\text{ and }\,\,\nabla^2_{{x}}F(t, x)=\int_0^t \nabla^2_{y}\psi(s, x_s, x(t)) ds,
			\]
			where $\nabla_{y}$ denotes the derivative in the last variable of $\psi$.
			In particular, if $ \nabla_{y}\psi,  \nabla^2_{y}\psi$ are Lipschitz continuous then $F$ satisfies the regularity properties of Assumption \ref{ass:sigma}. Note that we do not require any differentiability for $\psi$ in the path $x_s$, thus in general $F$ is not  Fr\'echet differentiable in the path.
		\end{enumerate}
	\end{example}
	
	\begin{remark}
		It is worth to mention that with minor technical modifications, the results of the article would hold if we replace the Lipschitz continuity assumption with an assumption of H\H older continuity in the metric $d_{\infty}$ (as in  \cite{ContKalin} and \cite{Kalin}). However, we chose to work in the Lipschitz continuous setting to avoid unnecessary complications.
	\end{remark}
	
	\subsection{Proof of the main result}
	
	\begin{theorem}[Existence of solutions]\label{chp3theorem:existence} 
		Let $\bm{X}=(X, \mathbb{X})\in \mathcal{C}^{p-var}([0,T],\Rr^d)$ be rough path over $X$, with $p\in [2, \sqrt{2}+1)$. Assume Assumptions \ref{ass:b} and \ref{ass:sigma} hold. Then for any $\xi\in C^{p-var}([0,t_0],\Rr^k)$ there exist $(Y,Y')\in\mathscr{D}^{p, q_p}_{X }([0,T],\Rr^k)$ a solution to \eqref{chp3ode} in the sense of Definition \ref{chp3def.odesol}.
	\end{theorem}
	
	Our proof follows an argument similar to the ones in \cite{frizhairer} for the H\H older setting, however unlike them, instead of a contraction argument, we use the Schauder fixed point theorem (\cite[Theorem 11.1]{GilbTrud}).

	\begin{proof}
		Without loss of generality we assume $t_0=0.$ We will prove that there exists a small enough time $T_0$ (depending only on $b,\,\sigma$ and $X$), such that the solutions exists on $[0,T_0],$ then one can apply the result on the intervals $[T_0,2T_0],\, [2T_0, 3T_0], \ldots$ until it reaches $T$.
		For any $(Y,Y')\in\mathscr{D}^{p,q_p}_{X }([0, T_0],\Rr^k)$, let us denote
		\[
		(\Xi,\Xi')=(\sigma(\cdot ,Y), \nabla  \sigma(\cdot ,Y) Y'),
		\]
		and define a mapping $\mathcal{M}_{T_0}\colon \mathscr{D}^{p,q_p}_{X }([0, T_0],\Rr^k)\to \mathscr{D}^{p,q_p}_{X }([0, T_0],\Rr^k)$ by
		\[
		\mathcal{M}_{T_0}(Y,Y'):=\left(\xi_0+\int_0^{\cdot} b(t, Y)dt+\int_0^{\cdot} \Xi\cdot d\bm{X} ,\,\,\, \Xi\right).
		\]
		The statement of the theorem is equivalent to the fact that $\mathcal{M}_{T_0}$ has a fixed point.
		To be able to use a compactness argument we will prove the existence  first in a larger space; take $r, p'$ such that $p<r<p'<\sqrt{2}+1$ and $q_{p'}<p.$ We denote $\kappa=r^{-1}, \nu=p^{-1}$, $\nu'=p'^{-1}$ and   $\beta_\kappa=q_r^{-1}, \beta_\nu=q_p^{-1}$, $\beta_{\nu'}=q_{p'}^{-1}.$
		
		We will prove the existence of a solution in $\mathscr{D}^{p',q_{p'}}_{X }$ and then argue that it is also in the initial space $ \mathscr{D}^{p,q_p}_{X }([0, T_0],\Rr^k)$. 
		We consider the subspace $A$ in $\mathscr{D}^{r,q_{r}}_{X }$, in the neighbourhood of the controlled path with constant Gubinelli derivative:
		$$t\mapsto (\underbrace{\xi+b(0,\xi) t+\sigma(0, \xi)X(t)}_{:=\bar{\xi}(t)},\, \sigma(0,\xi)),$$
		with $Y_0=\xi,\quad Y'_0=\sigma(0,\xi)$. To be precise, we introduce the following superadditive function on the intervals of $[0,T]$:
		\[
		\rho_{_\mathbf{X}}([t,s]):=|s-t|+\|X\|_{p,[t,s]}^p+\|\mathbb{X}\|_{\frac{p}{2},[t,s]}^{\frac p 2}.
		\]
		Now, we can define  the following H\H older seminorm associated to $\rho=\rho_{_\mathbf{X}}:$
		\[
		\|(Z, Z')\|_{\kappa, \beta, \rho}:=\|Z'\|_{\kappa, \rho, [0,T]}\,+\|R ^Z\|_{\beta, \rho, [0,T]},
		\]
		where $\|W\|_{\gamma, \rho, [0,T]}:=\sup_{0\leq t<s\leq T} \frac{|W_{t,s}|}{\rho([t,s])^{\gamma}},\, \gamma\in (0, 1]$ note that $$\|W\|_{p, [t,s]}\leq \|W\|_{_{\nu, \rho, [t, s]}}\, \rho([t,s])^{\nu},\, \text{ with }\nu=p^{-1}.$$
		Define the following subset of $\mathscr{D}^{p',q_{p'}}_{X }$:
		\begin{equation}
		A=\{(Y,Y')\in\mathscr{D}^{r,q_{r}}_{X }\colon Y_0=\xi,\quad Y'_0=\sigma(0,\xi),\quad \|(Y-b(0,\xi) t,Y')\|_{\kappa, \beta_{\kappa},\rho}\leq 1\} \subset \mathscr{D}^{p',q_{p'}}_{X },\quad \label{eq.setA}
		\end{equation}
		where  $\kappa=r^{-1}$ and $\beta_\kappa=q_r^{-1}$.

		It is easily checked that $A$ is a closed, convex subset of $\mathscr{D}^{p',q_{p'}}_{X }$. Moreover, by the following proposition $A$ is compact in $\mathscr{D}^{p',q_{p'}}_{X }$.
		\begin{proposition}\label{prop:compactness}
			Let $\alpha, \beta\in (0,1)$ and $p, q>1$ be such that $p>\frac 1 {\alpha},\, q>\frac 1 {\beta},$ and $\rho$ be a continuous superadditive function. Then for any $M>0$, the set
			\[
			\left\{
			(Y, Y')\in C([0,T], \Rr^d\times \Rr^{k\times d})\colon |Y(0)|\,+\,|Y'(0)|\,+\,\|Y, Y'\|_{\alpha, \beta, \rho}\leq M
			\right\},
			\]
			is compact in $\mathscr{D}^{p,q}_{X }.$
		\end{proposition}

		We divide the proof in two steps, where we check that the assumptions of Schauder theorem hold for $\mathcal{M}_{T}$ on the set $A\subset\mathscr{D}^{p',q_{p'}}_{X }$ and for small enough $T$. We already noted that  $A$ defined above is compact and convex in $\mathscr{D}^{p',q_{p'}}_{X }$, thus it remains to check the properties of $\mathcal{M}_{T}$.
		\vskip 0.5cm
		\begin{claim}[Invariance]
			There exist $\delta\in (0, 1)$ (depending only on global properties of $\xi, b, \text{ and }\sigma$) such that if $\rho([0,T_0])<\delta$, then the set $A$ defined by \eqref{eq.setA} is invariant under  $\mathcal{M}_{T_0}$: $$\mathcal{M}_{T_0}(A)\subset A.$$ 
		\end{claim}
		
		Let $(Y,Y')\in A$, we need to prove that $\mathcal{M}_{T_0}(Y,Y')\in A.$ First note that by definition of $A$ and the path $Y$
		\ba\nonumber
		|Y'_0|+\|(Y,Y')\|_{\kappa,\beta_{\kappa},\rho}\leq |Y'_0|+\|b(0, \xi) t\|_{\kappa, \rho}+\|(Y-b(0,\xi) t,Y')\|_{\kappa,\beta_{\kappa},\rho}\\
		\leq \|\sigma(0,\cdot)\|_{\infty}+\|b(0,\cdot)\|_{\infty}+1:=M.\nonumber
		\ea
		We obviously have
		\[
		\mathcal{M}_{T_0}(Y,Y')(0)=\xi,\, \mathcal{M}_{T_0}(Y,Y')'(0)=\Xi_0=\sigma(0, \xi),
		\]
		thus it remains to check that for small $\delta$, we have $\left\|\mathcal{M}_{T_0}(Y,Y')-(b(0, \xi)t, 0)\right\|_{\kappa, \beta_{\kappa}, \rho}\leq 1$. For that note
		\ba
		\left\|\mathcal{M}_{T_0}(Y,Y')-(b(0, \xi)t, 0)\right\|_{\kappa,\beta_{\kappa},\rho}\leq \left\|\int_0^{\cdot}b(t, Y)-b(0, \xi) dt\right\|_{q_{r}, \rho, [0,T_0]}\,+\,\left\|\int_0^{\cdot}\Xi d\bm{X}, \Xi\right\|_{\kappa,\beta_{\kappa},\rho}.\nonumber
		\ea
		For the first term on the right-hand side
		\[
		\left\|\int_0^{\cdot}b(t, Y)-b(0, \xi) dt\right\|_{\beta_{\kappa}, \rho, [0,T_0]}\leq T_0^{1-\beta_{\kappa}}\|b(\cdot, Y)-b(0, \xi)\|_{\infty}\leq C_{b,M}\delta^{1-\beta_{\kappa}}.
		\]
		where the last inequality follows from $d_{\infty}$ Lipschitz continuity of $b$ and $\|Y-\xi\|_\infty\leq \|Y\|_{r, [0, T_0]}\leq C_{M}$.
		To estimate the second term,  let $(Z, Z'):=\left(\int_0^{\cdot}\Xi d\bm{X}, \Xi\right)$, then the estimate \ref{eq:RZ} from the proof of Theorem \ref{chp02theorem:itegralcontrol} implies:
		\begin{align}\label{eq:RZ0}
		|R^Z_{t,s}|\leq \|\Xi'\|_{\infty}\|\mathbb{X}\|_{r/2,[t,s]}+
		C\left(\|X\|_{r,[t,s]}\|R^{\Xi}\|_{q_{r},[t,s]}+\|\mathbb{X}\|_{r/2,[t,s]}\|\Xi'\|_{r, [t,s]}\right)\nonumber\\
		\leq C\left(\|\Xi'\|_{\infty}+\|\Xi'\|_{r,[0, T_0]}\right)\|\mathbb{X}\|_{r/2,[t,s]}+C\|R^{\Xi}\|_{q_{r},[t,s]}\|X\|_{r,[t,s]}.
		\end{align}

		To estimate the first term in \eqref{eq:RZ0} note that from the first inequality of Theorem \ref{chp3theorem:actiononpaths}
		\eq{\label{chp3eq:Xixiineq}
			\|\Xi'\|_{\infty}+\|\Xi'\|_{r,[0, T_0]}\leq |\Xi'_0|+2\|\Xi'\|_{r,[0,T_0]}\leq\nonumber\\
			|\Xi'_0|+ C_{M,{\sigma}}(T_0+\|Y\|_{r,[0,T_0]}\, +\, \|Y'\|_{r, [0,T_0]})\leq \nonumber\\
			|\nabla  \sigma (0, \xi)\sigma(0, \xi)|+
			C_{M,{\sigma}}\left(1+|Y'_0|+\|(Y,Y')\|_{\kappa, \beta_{\kappa},\rho}\right)\leq C_{_{\xi, \sigma, M}},
		}
		For the second term in \eqref{eq:RZ0}, we use the second inequality of Theorem \ref{chp3theorem:actiononpaths}, which gives
		\eq{\label{chp3eq:Xixiineq2}\nonumber
			\|R^{\Xi}\|_{q_r, [t,s]}
			\leq C_{\sigma,M}(
			\|R^Y\|_{q_r, [t,s]}+ \|Y\|_{r, [t,s]}^{1+\kappa}+|s-t|)\nonumber\\
			\leq C_{\sigma,M}(
			\|R^Y\|_{_{\beta_{\kappa}}, \rho}\rho([t,s])^{\beta_{\kappa}}+ \|Y\|_{_{\kappa, \rho}}^{1+\kappa} \rho([t,s])^{\beta_{\kappa}}+\rho([t,s]))\leq C_{_{\sigma, M}}\,\rho([t,s])^{\beta_{\kappa}}.
		}
		Consequently, from \eqref{eq:RZ0}
		\begin{align*}
		|R^Z_{t,s}|\leq  C_{_{\xi, \sigma, M}}\|\mathbb{X}\|_{r/2,[t,s]}+C_{_{\sigma, M}}\,\|X\|_{r,[t,s]}\rho([t,s])^{\beta_{\kappa}}\leq C_{_{\xi, \sigma, M}}\left( \rho([t,s])^{2\kappa} +\rho([t,s])^{\kappa+\beta_{\kappa}} \right),
		\end{align*}
		\eq{\label{eq:RZbdd}
			\|R^Z\|_{\beta_{\kappa}, \rho}\leq C_{_{\xi, \sigma, M}}\delta^{2\kappa-\beta_{\kappa}}.
		}
		Using Lemma \ref{chp3theorem:controll}, as in the proof of \eqref{chp3eq:DFYnu}, we get
		$$|\Xi_{t,s}|\leq \|\Xi\|_{r,[t,s]}=\|{\sigma}(\cdot, Y)\|_{r,[t,s]}\leq C_{_{\sigma, M}}\Big(|s-t|+ \|Y\|_{r,[t,s]}\Big).$$ 
		From the identity $(Y)_{t,s} =Y'(t) (X)_{t, s}\,+\, R^Y_{t,s}$ through the chain of inequalities:
		\begin{align*}
		\|Y\|_{r, [t,s]}\nonumber\leq \|Y'\|_{\infty}\|X\|_{p,[t,s]}+\|R^{Y}\|_{q_{r},[t, s]}
		\leq \rho([t,s])^{\min\{\nu, \beta_{\kappa}\}}\Big( \|Y'\|_{\infty}\, +\,\|R^{Y}\|_{\beta_{\kappa}, \rho}\Big)\nonumber\\
		\leq \Big( |Y'_0|+\|(Y,Y')\|_{\kappa,\beta_{\kappa}, \rho}\Big)\rho([t,s])^{\min\{\nu, \beta_{\kappa}\}}
		\leq C_{_{\sigma, M}}\rho([t,s])^{\min\{\nu, \beta_{\kappa}\}}.
		\end{align*}
		The previous two inequalities yield
		\begin{align}\label{chp3eq:ineqYk}
		\|\Xi\|_{\kappa, \rho}\leq C_{_{\sigma, M}}\delta^{\min\{\nu, \beta_{\kappa}\}-\kappa}
		\end{align}
		Finally \eqref{eq:RZbdd} and \eqref{chp3eq:ineqYk} give
		\ba
		\left\|\mathcal{M}_{T_0}(Y,Y')-(b(0, \xi)t, 0)\right\|_{\kappa,\beta_{\kappa}, \rho}\leq \left\|\int_0^{\cdot}b(t, Y)-b(0, \xi) dt\right\|_{\beta_{\kappa}, \rho}\,+\,\left\|\int_0^{\cdot}\Xi d\bm{X}, \Xi\right\|_{\kappa,\beta_{\kappa}, \rho}\nonumber\\
		\leq C_{_{b, M}} \delta^{1-\beta_{\kappa}}+C_{_{\xi, \sigma, M}}\delta^{\min\{\nu,\beta_{\kappa}\}-\kappa}.\nonumber
		\ea
		It remains to take $T_0$ small enough so that $C_{_{b, M}} \delta^{1-\beta_{\kappa}}+C_{_{\xi, \sigma, M}}\delta^{\min\{\nu,\beta_{\kappa}\}-\kappa}\leq 1,$ to get $\mathcal{M}_{T_0}(Y,Y')\in A$.
		
		\begin{claim}[Continuity]
			The map
			$$
			\mathcal{M}_{T_0}\colon \mathscr{D}^{p', q_{p'}}_{X }\to \mathscr{D}^{p', q_{p'}}_{X }.
			$$ is continuous.
		\end{claim}
		Let $(Y_1,Y_1'),\, (Y_2,Y_2')\in \mathscr{D}^{p',q_{p'}}_{X }$, and $\Delta(s)=\sigma(s,Y_{1})-\sigma(s,Y_{2}),\, \Delta'(s)=\nabla\sigma(s,Y_{1})-\nabla\sigma(s,Y_{2})$. Let $R, M>0$ be such that
		\[
		|Y_1(0)|+\|Y_1\|_{p',[0,T_0]} +\|R^{Y_1}\|_{q_{p'},[0,T_0]},\,|Y_2(0)|+\|Y_2\|_{p', [0,T_0]}\leq R,
		\]
		and
		\[
		\|Y_1-Y_2\|_{\infty}+\|Y_1-Y_2\|_{p,[0,T_0]}\leq M.
		\]
		We can estimate the distance between the values $\mathcal{M}_{T_0}(Y_1,Y_1')$ and $\mathcal{M}_{T_0}(Y_2,Y_2')$  by Theorem \ref{chp02theorem:itegralcontrol}

		\eq{\label{chp3eq:Mtbound}
			\left\|\mathcal{M}_{T_0}(Y_1,Y_1')-\mathcal{M}_{T_0}(Y_2,Y_2')\right\|_{p', q_{p'}, X}=\left\|\int_0^{\cdot}\Delta d\mathbf{X}, \Delta\right\|_{p', q_{p'}, X}\nonumber\leq
			\\\|\Delta\|_{p',[0,T_0]}+C\left(\|X\|_{p'}\|R^{\Delta}\|_{q_{p'},[0,T_0]}+\|\Delta'\|_{p',[0,T_0]}\|\mathbb{X}\|_{p'/2,[0,T_0]}\right)\nonumber\\
			\leq 
			\|\Delta\|_{p'}+C_{\mathbf{X}}\left(|\Delta'_0|+\|(\Delta,\Delta')\|_{p', q_{p'}, X}\right)
		}
		For the first term in \ref{chp3eq:Mtbound}, we use the Corollary \ref{chp3Cor:controlcont} for $\sigma$
		\begin{align}
		\label{chp3eq:Deltanubound}
		\|\Delta\|_{p',[0,T_0]}=\|{\sigma}(\cdot, Y_1)-{\sigma}(\cdot, Y_2)\|_{p',[0,T_0]}
		\leq C_{\sigma,M, R} \left(\|Y_1-Y_2\|_{\infty}^{\nu'}\,+\,\|Y_1-Y_2\|_{p', [0, T_0]}^{\nu'}\right).
		\end{align}
		For $\nabla\sigma$ the same corollary provides
		\begin{align*}
		\|\nabla{\sigma}(\cdot, Y_1)-\nabla{\sigma}(\cdot, Y_2)\|_{p',[0,T_0]}
		\leq C_{\nabla  \sigma, M, R}\left(\|Y_1-Y_2\|_{\infty}^{\nu'}\,+\,\|Y_1-Y_2\|_{p', [t,s]}^{\nu'}\right)
		\end{align*}
		Also by $\sigma\in Lip(\Lambda_T^d, d_{\infty})$ \begin{align*}\|\nabla\sigma(\cdot, Y_1)\|_{\infty} \leq |\nabla \sigma(0, \xi)| +C_{\sigma}\|Y_1-\xi\|_{\infty}\le C_{\xi, \sigma, R},\end{align*} 
		From previous two inequalities, we conclude
		\begin{align}
		\label{eq:Deltanu'}
		\|\Delta'\|_{p', [0,T_0]}=\|\nabla {\sigma}(\cdot, Y_1)Y'_1-\nabla {\sigma}(\cdot, Y_2)Y'_2\|_{p', [0,T_0]}
		\nonumber\\\leq \|\nabla\sigma(\cdot, Y_1)\|_{\infty}\|Y'_1-Y'_2\|_{p', [0,T_0]}
		+\|\nabla {\sigma}(\cdot, Y_1)-\nabla {\sigma}(\cdot, Y_2)\|_{p', [0,T_0]}\|Y'_2\|_{\infty}\nonumber\\
		\leq  C_{\xi, \sigma, M, R } \Big(\|Y_1-Y_2\|_{\infty}^{\nu'}\,+\,\|Y'_1-Y'_2\|_{p', [0,T_0]}\,+\, \left\|Y_1-Y_2\right\|_{p', [0,T_0]}^{{\nu'}}\Big)\nonumber
		\\
		\leq C_{\xi, \sigma, M, R }\,\Big(|(Y_1-Y_2)(0)|^{\nu'}+\left\|Y_1-Y_2,Y'_1-Y'_2\right\|^{\nu'}_{p', q_{p'}, X}\Big)
		\end{align}
		where we have used
		\[
		\|Y_1-Y_2\|_{p', [0,T_0]}\leq C_{T_0, X} \Big(|(Y_1-Y_2)(0)|+\left\|Y_1-Y_2,Y'_1-Y'_2\right\|_{p', q_{p'}, X}\Big).
		\]
		Finally, from
		\[
		R^{\Delta}_{t,s}=\mathcal{R}^{\sigma,Y_1}_{t,s}-\mathcal{R}^{\sigma,Y_2}_{t,s}+\nabla{\sigma}(t, Y_1)R_{t,s}^{Y_1}-\nabla{\sigma}(t, Y_2)R_{t,s}^{Y_2}
		\]
		we have
		\begin{align*}
		|R^{\Delta}_{t,s}|\leq |\mathcal{R}^{\sigma,Y_1}_{t,s}-\mathcal{R}^{\sigma,Y_2}_{t,s}|\,+\,|\nabla{\sigma}(t, Y_1)-\nabla{\sigma}(t, Y_2)|\,|R_{t,s}^{Y_1}|\, + \,|\nabla{\sigma}(t, Y_1)|\,|R_{t,s}^{Y_1}-R_{t,s}^{Y_2}|
		\end{align*}
		from Lipschitz continuity of $\nabla\sigma$, the inequality $\|\nabla\sigma(\cdot, Y_1)\|_{\infty} \leq C_{\xi, R}$ (obtained above) and Lemma \ref{chp3theorem:contrcont}
		\begin{align}\label{eq:RDeltab}
		\|R^{\Delta}\|_{q_{\nu'}}\leq  C_{\sigma, T_0, R}(\|Y_1-Y_2\|^{\nu'}_{\infty}+\|Y_1-Y_2\|^{\nu'}_{p', [0,T_0]})\,+\,
		C_{ \sigma,T_0} \|Y_1-Y_2\|_{\infty} \|R^{Y_1}\|_{q_{\nu'}}\,\nonumber\\
		+\,C_{\sigma}\left\|Y_1-Y_2,Y'_1-Y'_2\right\|_{p', q_{p'}, X}
		\leq C_{ \xi, \sigma, T_0, R }\,\,  \left(\left\|Y_1-Y_2,Y'_1-Y'_2\right\|^{\nu'}_{p', q_{p'}, X}\right).
		\end{align}
		Combining \eqref{eq:RDeltab} and \eqref{eq:Deltanu'} and using
		\[
		\|Y_1-Y_2\|_{\infty}+\|Y_1-Y_2\|_{p', [0,T_0]}\leq C_{T_0}\left\|Y_1-Y_2,Y'_1-Y'_2\right\|_{p', q_{p'}, X}
		\]
		we get
		\begin{align*}
		\|(\Delta,\Delta')\|_{p', q_{p'}, X}= \|\Delta'\|_{p', [0,T_0]}+ \|R^{\Delta}\|_{q_{\nu'}}\\\leq C_{\xi, \sigma,T_0, R }\quad  \Big(|(Y_1-Y_2)(0)|^{\nu'}+\left\|Y_1-Y_2,Y'_1-Y'_2\right\|^{\nu'}_{p', q_{p'}, X}\Big).
		\end{align*}
		Thus \eqref{chp3eq:Mtbound} and \eqref{chp3eq:Deltanubound} yield
		\begin{align*}
		\left\|\mathcal{M}_{T_0}(Y_1,Y_1')-\mathcal{M}_{T_0}(Y_2,Y_2')\right\|_{p', q_{p'}, X}\\
		\leq C_{\xi,\sigma,T_0, R }\quad  \Big(|(Y_1-Y_2)(0)|^{\nu'}+\left\|Y_1-Y_2,Y'_1-Y'_2\right\|^{\nu'}_{p', q_{p'}, X}\Big)\\
		\leq 
		C_{\xi,\sigma,T_0, R } \left\|Y_1-Y_2,Y'_1-Y'_2\right\|^{\nu'}_{\mathscr{D}^{p', q_{p'}}_ X}.
		\end{align*}
		hence $\mathcal{M}_{T_0}$ is continuous in $\mathscr{D}^{p', q_{p'}}_{X}$.
		
		\vskip0.5cm
		We have proved that $\mathcal{M}_{T_0}\colon \mathscr{D}^{p', q_{p'}}_{X}\to \mathscr{D}^{p', q_{p'}}_{X}$ is continuous, $\mathcal{M}_{T_0}(A)\subset A$, and $A\subset \mathscr{D}^{p', q_{p'}}_{X}$ is a compact, convex subset. Thus by Schauder  fixed point theorem, $\mathcal{M}_{T_0}$ has a fixed point $(Y,Y')\in\mathscr{D}^{p', q_{p'}}_{X}$.
		
		To conclude the proof of the theorem it remains to prove that $(Y_s,Y'_s)\in\mathscr{D}^{p,q_p}_{X}$. 
		Indeed, from the representation $Y_{t,s}=Y'_tX_{t,s}+R^Y_{t,s}$, $X\in C^{p-var}([0,T],\Rr)$ and $q_{p'}<p$ it follows that
		\[
		\|Y\|_{p, [0, T_0]}\leq \|Y'\|_{\infty}    \|X\|_{p,[0,T_]}\,+\,\|R^Y_{t,s}\|_{q_{p'},[0,T_0]}<+\infty,
		\]
		hence $Y\in C^{p-var}([0,T_0],\Rr)$. Now, using the fixed point property 
		\[
		(Y,Y')=\left(\xi_0+\int_0^{\cdot} b(t, Y)dt+\int_0^{\cdot} \Xi\cdot d\bm{X} ,\,\,\, \Xi\right),\, (\Xi,\Xi')=(\sigma(\cdot ,Y), \nabla  \sigma(\cdot ,Y) Y'),
		\]
		and Corollary \ref{chp3Cor:controlcont}, we get $Y'=\sigma(\cdot, Y)\in C^{p-var}$.
		Next, by Theorem \ref{chp02theorem:itegralcontrol}
		\[
		\left|R^Y_{t,s}\right|\nonumber 
		\leq \|b\|_{\infty}|s-t|\,+\,\|\Xi'\|_{\infty}|\mathbb{X}_{t,s}|\,+\, C\left(\|{R}^{\Xi}\|_{q_p',[t,s]}\|X\|_{p',[t,s]}+\|\Xi'\|_{p',[t,s]}\|{\mathbb{X}}\|_{p'/2,[t,s]}\right)\nonumber.
		\]
		Using that $1/p'+1/q_{p'}>1$, $p'<3$ and the properties of superadditive functions the above yields
		\[
		\left|R^Y_{t,s}\right|\leq \|\Xi'\|_{\infty}|\mathbb{X}_{t,s}|\,+\,\omega([t,s]),
		\]
		where $\omega$ is a superadditive interval function. Since $p/2< q_{p}$ and $\|X\|_{p,[0,T]}<+\infty$, we conclude that $\|R^Y\|_{q_p, [0,T_0]}<+\infty$, therefore $(Y,Y')\in\mathscr{D}^{p,q_p}_{X }.$
	\end{proof}

	\begin{remark}
		Our proofs suggest that  the results would still hold under the following regularity assumption on the coefficient $\sigma$:\\
		
		\begin{hyp}
			There exist a non-anticipative functional $\sigma'\colon\Lambda^k_T\to \mathcal{L}(\Rr^k,\mathcal{L}(\Rr^d,\Rr^k))$ such that
			\begin{itemize}
				\item $\sigma, \sigma'$ are continuous in the p-variation norm;
				there exist a modulus of continuity $\rho\colon \Rr_+\to \Rr_+,\, \rho(r)\xrightarrow[r\to 0]{} 0$:
				\begin{align*}
				\|\sigma(\cdot, Y_1)-\sigma(\cdot, Y_2)\|_{{C}^{p-var}}\leq C_{\sigma, T}\,\,\rho \left(\| Y_1- Y_2)\|_{{C}^{p-var}}\right),\\
				\|\sigma'(\cdot, Y_1)-\sigma'(\cdot, Y_2)\|_{{C}^{p-var}}\leq C_{\sigma, T}\,\,\rho \left(\| Y_1- Y_2)\|_{{C}^{p-var}} \right).
				\end{align*}
				
				\item  $R^{\sigma, \sigma'}_{t,s}(Y):=\sigma(s, Y)\,-\,\sigma(t, Y)\,-\,\sigma'(t,Y)(Y(s)-Y(t))$ satisfies
				\begin{itemize}
					\item[a)] $\|R^{\sigma, \sigma'}(Y)\|_{q_p, [t,s]}\leq C_{\sigma, T}\,\, \rho\left( \|Y\|_{\infty, [0,s]}\,+\,\|Y\|_{p, [t,s]}\right), \forall Y\in C^{p-var}([0,T], \Rr^k)$
					
					\item [b)] $\|R^{\sigma, \sigma'}(Y_1)\,-\,R^{\sigma, \sigma'}(Y_2)\|_{q_p, [t,s]}\leq C_{\sigma, T}\,\, \rho\left(\|Y_1\,-\,Y_2\|_{{C}^{p-var}}\right), \forall Y_1,\, Y_2\in C^{p-var}([0,T], \Rr^k)$
				\end{itemize}
				for some $q_{p}\in (0,1)$ with $q_{p}^{-1}+p^{-1}>1$.
			\end{itemize}
		\end{hyp}
	\end{remark}
	
	\newpage
	\section*{Appendix: {Proof of Proposition \ref{prop:compactness} }}
	In this section we present the proof of the Proposition \ref{prop:compactness}:
	\begin{proof}
		It is enough to show that any sequence $(Y_n, Y'_n)\in C([0,T], \Rr^d\times \Rr^{k\times d})$ satisfying 
		\[
		|Y_n(0)|\,+\,|Y_n'(0)|\,+\,\|Y_n, Y'_n\|_{\alpha, \beta, \rho}\leq M
		\]
		has a convergent subsequence in $\mathscr{D}^{p,q}_{X }.$ For this note that since
		\eq{\label{eq:Yineq}
			|Y_n(s)-Y_n(t)|\leq C_{M, T}\rho([s,t])^{\alpha}\nonumber\\
			|Y'_n(s)-Y'_n(t)|\leq M\rho([s,t])^{\alpha}
		}
		by Arzela-Ascoli theorem we can assume that $(Y_n, Y'_n)\to (Y, Y')\in C([0,T], \Rr^d\times \Rr^{k\times d})$ uniformly. From \eqref{eq:Yineq} and since also 
		\eq{\label{eq:RYineq}\nonumber
			|R^{Y_n}_{t,s}|\leq M\rho([s,t])^{\beta},
		}
		we conclude $\|\tilde{Y}'\|_{\alpha, \rho}\leq M,\, \|R^{{Y}}\|_{\beta, \rho}\leq M.$ 
		\vskip5pt
		Let now $\tilde{Y}_n:=Y_n-Y,\, \tilde{Y}'_n:=Y'_n-Y'$ and $\alpha':=p^{-1}<\alpha,\, \beta':=q^{-1}<\beta$. We have that $\|\tilde{Y}_n\|_{\infty}, \|\tilde{Y}_n'\|_{\infty}\to 0$ and consequently from $\|R^{{Y}}\|_{\infty}:=Y_{t,s}-Y'(t)X_{t,s}$  $$\|R^{\tilde{Y}_n}\|_{\infty}\leq C\left(\|\tilde{Y}_n\|_{\infty}\,+\,\|\tilde{Y}'_n\|_{\infty}\|X\|_{\infty}\right)\to 0.$$
		Since
		$$\|\tilde{Y}_n'\|_{\alpha, \rho}\leq \|Y_n'\|_{\alpha, \rho}+\|\tilde{Y}'\|_{\alpha, \rho}\leq 2 M,\, \|R^{\tilde{Y}_n}\|_{\beta, \rho}\leq \|R^{{Y}_n}\|_{\beta, \rho}\,+\,\|R^{{Y}}\|_{\beta, \rho}\leq 2 M,$$
		we get
		\[
		\|\tilde{Y}_n'\|_{\alpha', \rho}\leq\left(2 \|\tilde{Y}_n'\|_{\infty}\right)^{1-\frac{\alpha'}{\alpha}} \left(\|\tilde{Y}_n'\|_{\alpha, \rho}\right)^{\frac{\alpha'}{\alpha}}\leq C_M \left(\|\tilde{Y}_n'\|_{\infty}\right)^{1-\frac{\alpha'}{\alpha}}\to 0,
		\]
		\[
		\|R^{\tilde{Y}_n}\|_{\alpha', \rho}\leq\left( \|R^{\tilde{Y}_n}\|_{\infty}\right)^{1-\frac{\beta'}{\beta}} \left(\|R^{\tilde{Y}_n}\|_{\beta, \rho}\right)^{\frac{\beta'}{\beta}}\leq C_M\left( \|R^{\tilde{Y}_n}\|_{\infty}\right)^{1-\frac{\beta'}{\beta}}\to 0,
		\]
		Convergence in $\mathscr{D}^{p,q}_{X }$ follows from the inequalities
		$$\|Y'\|_{p, [0,T]}\leq \|Y'\|_{_{\alpha', \rho}}\, \rho([0,T])^{\alpha'},\, \|R^Y\|_{q, [0,T]}\leq \|R^Y\|_{_{\beta', \rho}}\, \rho([0,T])^{\beta'}.$$
	\end{proof}

	\clearpage
	\bibliographystyle{plain}
	\bibliography{references-3}

\end{document}